\documentclass[a4paper,oneside,10pt,oldfontcommands]{article}
\usepackage[a4paper, total={426pt, 674pt}]{geometry}

\newcommand{\langue}{anglais}	

\usepackage{ifthen}
\ifthenelse{\equal{\langue}{francais}}{
	\usepackage[french]{babel}}
	{\ifthenelse{\equal{\langue}{anglais}}{\usepackage[english]{babel}}{}}
\usepackage[T1]{fontenc}
\usepackage[utf8]{inputenc}
\usepackage[babel,german=guillemets]{csquotes}
\usepackage{eso-pic}

\usepackage[backend=biber, style=alphabetic, sorting=nyt]{biblatex} 
\AtEveryBibitem{\clearfield{language}}
\AtEveryBibitem{\clearfield{url}}
\AtEveryBibitem{\clearfield{issn}}
\usepackage{lmodern} 
\usepackage{enumitem}
\usepackage{ifthen}
\ifthenelse{\equal{\langue}{francais}}{
	\frenchbsetup{StandardLists=true}}
	{}

\usepackage{graphicx} 
\usepackage{subcaption}

\usepackage{amsmath}
\usepackage{amssymb}
\usepackage{amsthm}
\usepackage{amsfonts}
\usepackage{MnSymbol}
\usepackage{bbold}
\usepackage{mathtools}
\usepackage{stmaryrd}
\usepackage{tikz-cd}
\usepackage{hyperref}
\usepackage{multirow}

\ifthenelse{\equal{\langue}{francais}}{
	\newcommand{\theoremenom}{Théorème}
	\newcommand{\propositionnom}{Proposition}
	\newcommand{\lemmenom}{Lemme}
	\newcommand{\corollairenom}{Corollaire}
	\newcommand{\definitionnom}{Définition}
	\newcommand{\remarquenom}{Remarque}
	\newcommand{\exemplenom}{Exemple}
	\newcommand{\conjecturenom}{Conjecture}
}{\ifthenelse{\equal{\langue}{anglais}}{
	\newcommand{\theoremenom}{Theorem}
	\newcommand{\propositionnom}{Proposition}
	\newcommand{\lemmenom}{Lemma}
	\newcommand{\corollairenom}{Corollary}
	\newcommand{\definitionnom}{Definition}
	\newcommand{\remarquenom}{Remark}
	\newcommand{\exemplenom}{Example}
	\newcommand{\conjecturenom}{Conjecture}
}{}}
	
\newtheorem{theoreme}{\theoremenom}
\newtheorem{proposition}[theoreme]{\propositionnom}
\newtheorem{lemme}[theoreme]{\lemmenom}

\newtheorem{remarque}[theoreme]{\remarquenom}

\usepackage{thmtools}
\usepackage{thm-restate}

\makeatletter
\def\cleartheorem#1{%
    \expandafter\let\csname#1\endcsname\relax
    \expandafter\let\csname c@#1\endcsname\relax
}
\makeatother
\newcommand{\compteurThm}{1}
\newcounter{annexe}

\newcommand{\R}{\mathbb{R}}

\begin{document}

\pagestyle{empty} 


\title{Sharp Poincaré inequality for $W_0^{2,p}$}
\author{
Roméo Leylekian
\footnote{Grupo de Física Matemática, Instituto Superior Técnico, Universidade de Lisboa - \textbf{email:} romeo.leylekian@tecnico.ulisboa.pt}
}
\date{} 
\maketitle


\begin{abstract}
The first eigenvalue of the $p$-bilaplacian over a domain $\Omega$ is the inverse of the Poincaré constant for $W_0^{2,p}(\Omega)$. In dimension two, we prove that it is minimised by disks under area constraint, thereby obtaining a sharp uniform Poincaré inequality.
\end{abstract}




\pagestyle{plain} 


\section{Introduction}
\newcommand{\B}{B}

Poincaré inequality states that the $L^p$ norm of a function vanishing at the boundary of a domain is controlled by the $L^p$ norm of its gradient. A natural higher-order version is the following:
\begin{equation}\label{eq:poincaré}
\int_\Omega |u|^p\leq C_p(\Omega)\int_\Omega|\Delta u|^p,\qquad\forall u\in W_0^{2,p}(\Omega).
\end{equation}
Here $\Omega$ is a bounded open set of $\R^d$ and $W_0^{2,p}(\Omega)$ denotes the subspace of $W^{2,p}(\Omega)$ made of functions vanishing at the boundary, together with their gradient -- formally the closure of compactly supported functions in $W^{2,p}(\Omega)$. The purpose of the present article is to investigate the best possible constant in~\eqref{eq:poincaré}. We will establish that in dimension two and for $1<p<\infty$,
$$
C_p(\Omega)\leq C_p(\B),
$$
where $\B$ is a disk of same area as $\Omega$. The value of $C_p(\B)$ cannot be expressed explicitly -- this already happens with the standard Poincaré inequality, even in the linear case $p=2$, where the constant depends on the first zero of a Bessel function -- but it can be numerically evaluated, see Figure~\ref{fig:numerique}.
For an arbitrary domain $\Omega$, the constant $C_p(\Omega)$ can also be related to a nonlinear eigenvalue problem. Indeed, it is clearly the inverse of
\begin{equation}\label{eq:eigenvalue}
\Gamma_p(\Omega):=\inf_{u\in W_0^{2,p}(\Omega)\setminus\{0\}}\frac{\int_\Omega|\Delta u|^p}{\int_\Omega |u|^p}.
\end{equation}
When the domain is smooth enough, say $C^{1,1}$, the infimum is achieved, by reflexivity of $W_0^{2,p}(\Omega)$ and elliptic estimates. The corresponding function satisfies the system
\begin{equation}\label{eq:eigenvalue equation}
\left\{
\begin{array}{rcll}
\Delta_{2,p} u & = & \Gamma_p(\Omega) |u|^{p-2}u & \text{in }\Omega, \\
u & = & 0 & \text{on }\partial \Omega, \\
\partial_n u & = & 0 & \text{on }\partial \Omega.
\end{array}
\right.
\end{equation}
Here $\partial_n$ denotes the normal derivative while the symbol $\Delta_{2,p}$ stands for the nonlinear operator of fourth-order defined by
$$
\Delta_{2,p} u=\Delta(\Delta u|\Delta u|^{p-2}),
$$
and usually called \emph{$p$-bilaplacian}. For $p=2$ we recover the standard bilaplacian. In view of~\eqref{eq:eigenvalue equation} we say, through misuse of language, that $\Gamma_p(\Omega)$ is the first eigenvalue and that $u$ is a first eigenfunction of the $p$-bilaplacian. We point out that all of this fails when $p=1$~\cite{parini-ruf-tarsi15} whereas for $p=\infty$ equation~\eqref{eq:eigenvalue equation} has to be understood in a generalised sense~\cite{katzourakis-parini}.

Proving that $C_p(\Omega)\leq C_p(B)$ thus amounts to showing that, among domains of given measure, the ball minimises the first eigenvalue of the $p$-bilaplacian. For $p=2$, this corresponds to an old conjecture (1877) of Rayleigh~\cite{rayleigh}. Rayleigh's conjecture was proved by Nadirashvili in dimension two~\cite{nadirashvili} and then generalised by Ashbaugh and Benguria in dimension three~\cite{ashbaugh-benguria}; but it remains open in dimension $d\geq4$ in spite of partial results, including~\cite{szego,ashbaugh-laugesen}. For $p=1$ and $p=\infty$, the result holds respectively in dimension two and in any dimension, in view of the works of Parini et al.~\cite{parini-ruf-tarsi15,katzourakis-parini}. For other values of $p$, nothing seems to be known (see however~\cite[section 7.2]{leylekian} for related inequalities involving the $p$-bilaplacian). The next result shows that, in dimension two, the spectral inequality actually holds for any $p$.

\begin{theoreme}\label{thm:Faber-Krahn}
Let $\Omega$ be a planar bounded open set. Then, for all $1<p<\infty$,
\begin{equation}\label{eq:FK}
\Gamma_p(\Omega)\geq\Gamma_p(\B),
\end{equation}
where $\B$ is a disk of same area as $\Omega$. Furthermore, in case of equality, $\Omega$ coincides with a disk, as long as it admits an eigenfunction.
\end{theoreme}

A nontrivial byproduct of our proof is the following description of the eigenspace on the disk. Note that we do not claim uniqueness of the eigenfunction -- see also Remark~\ref{rmq:unicité}.

\begin{theoreme}\label{thm:cas de la boule}
For the disk any minimiser of~\eqref{eq:eigenvalue} is, up to normalisation, radially symmetric, positive, and decreasing.
\end{theoreme}

\begin{figure}[h!]
\centering
\includegraphics[width=.8\textwidth]{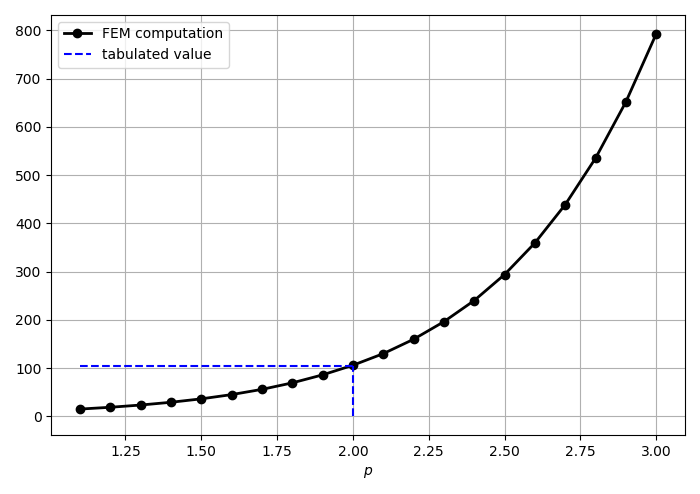}
\caption{Numerical estimation of $\Gamma_p(\mathbb{D})$ for $1<p\leq3$, where $\mathbb{D}$ is the unit disk, obtained by means of an AI generated Finite Element Method with second order elements. The eigenvalue is approached by minimising the Rayleigh quotient~\eqref{eq:eigenvalue}, restricted to radial functions by Theorem~\ref{thm:cas de la boule}. For $p=2$, the result is compared to the tabulated value $\Gamma_2(\mathbb{D})\simeq 104.363$~\cite{ashbaugh-benguria}.}
\label{fig:numerique}
\end{figure}

The challenge with higher-order problems is that the first eigenfunction may change sign, even on smooth convex domains~\cite{kondratiev-kozlov-mazya}. If instead the variational formulation~\eqref{eq:eigenvalue} were restricted to nonnegative functions, then~\eqref{eq:FK} would immediately follow from a traditional symmetrisation procedure introduced by Talenti~\cite{talenti81} and would extend to higher dimension as well. See also~\cite{szego} and \cite[Theorem~1.2]{parini-ruf-tarsi15}.

In the general case, proving Theorem~\ref{thm:Faber-Krahn} reduces, after Talenti's symmetrisation, to the study of an auxiliary problem -- called \emph{two-ball} problem, since it is posed on two disjoint balls $B_a$ and $B_b$, with unknown radii constrained by $a^d+b^d=1$. Those balls are reminiscent of the positive and negative parts of the eigenfunction prior to symmetrisation. The proof is completed if one can show that the two-ball problem is minimal when one ball shrinks while the other covers all the volume, namely when $(a,b)$ is $(0,1)$ or $(1,0)$. In the linear case $p=2$, the two-ball problem can be expressed in terms of Bessel functions and studied accordingly. By this mean, the minimality of $(0,1)$ and $(1,0)$ was obtained in dimensions two and three~\cite{ashbaugh-benguria}. In higher dimension, the same approach showed that the two-ball problem is minimal when the two balls have equal volume~\cite{ashbaugh-laugesen}, equivalently when $a=b=2^{-1/d}$, thereby ruling out this avenue.

In the nonlinear setting, the previous computational approach is not available. Therefore, we follow a more abstract method (historically preceding the computational one) proposed in~\cite{nadirashvili}. Morally, the derivative of the two-ball problem with respect to $a$ is estimated, giving rise to some condition -- see~\eqref{eq:condition pour la derivee} -- between the solutions of a nonlinear fourth-order equation on $B_a$ and $B_b$. In dimension two, a pointwise relation can be proved between those solutions, leading to~\eqref{eq:condition pour la derivee} and hence to the result. In higher dimension, the pointwise relation fails. However, we expect condition~\eqref{eq:condition pour la derivee} to hold in more general regimes: typically, it must hold for $p=2$ and $d=3$; but we leave this question open here. Interestingly, we do not prove the pointwise relation as in~\cite{nadirashvili} -- there it follows from studying the derivative of yet another auxiliary problem, see~\cite[Lemma~3.8]{nadirashvili}. In the nonlinear setting, differentiability properties are more delicate, hence we resort instead to a direct comparison based on a new maximum principle result, see Lemma~\ref{lemme:principe du max} in Appendix.

\noindent\textbf{Note.} A couple of days ago, two articles were published by D. Bucur and R. Laugesen on arXiv~\cite{bucur-laugesen1, bucur-laugesen2}. In~\cite{bucur-laugesen2}, the authors obtain a general isoperimetric inequality for nonlinear Rayleigh quotients, in dimension two. There main result contains our Theorem~\ref{thm:Faber-Krahn} as a particular case. The proof uses a new Talenti-type comparison principle~\cite{bucur-laugesen2}, valid for signed functions in dimension two, and allowing them to bypass the two-ball problem. A similar idea had been developed in~\cite[chapter~7]{leylekian}, though with less success. The proof proposed in the present article differs and uses the (dimension independant) reduction to the two-ball problem. In principle, there is no obstruction in adapting this approach in higher dimension -- yet new ideas are required to establish~\eqref{eq:condition pour la derivee} -- while Bucur-Laugesen's method seems limited to the planar setting.

\noindent\textbf{AI statement.} Except for Figure~\ref{fig:numerique} and possible language editing, no AI assistance was used in developing the scientific content of this document.

\section{Proofs}

As an attempt to clarify the limitation to the planar setting, we consider in the following the general case of a domain $\Omega$ of $\R^d$, and point out explicitly when the restriction $d=2$ is required.

\subsection{Reduction to the two-ball problem}

The first step is to extract the nonlinear two-ball problem.

\begin{proposition}\label{prop:extraction pb 2 boules}
Let $\Omega$ be a bounded open subset of\/ $\R^d$ and $1<p<\infty$. There exist two balls $B_a$ and $B_b$ with $|B_a|+|B_b|=|\Omega|$ (one of the balls may be empty), such that
$$\Gamma_p(\Omega)\geq I_{a,b},$$
where
\begin{equation}
I_{a,b}:=\inf_{v,w}\frac{\int_{B_a}|\Delta v|^p+\int_{B_b}|\Delta w|^p}{\int_{B_a}|v|^p+\int_{B_b}|w|^p}.
\end{equation}
When $a,b\neq0$, the infimum is taken among radial functions $v\in W^{2,p}(B_a)$ and $w\in W^{2,p}(B_b)$ such that
$$
v(a)=w(b)=0,\qquad\int_{\partial B_a}\partial_nv=\int_{\partial B_b} \partial_n w.
$$
When $a=0$ or $b=0$, the infimum is taken with radial $v\in W_0^{2,p}(B_a)$ and $w\in W_0^{2,p}(B_b)$.
\end{proposition}

\begin{proof}
The proof is a classical application of Talenti's comparison principle. Let $u\in W_0^{2,p}(\Omega)$ and set $\Omega_\pm:=\{\pm u>0\}$. We define $v$ over $B_a$ by the equation $-\Delta v=(-\Delta u)|_{\Omega_+}^*$, with boundary data $v=0$, and $w$ over $B_b$ by $-\Delta w=(\Delta u)|_{\Omega_-}^*$, also with $w=0$ at the boundary. Here $^*$ denotes the (signed) Schwarz symmetrisation~\cite[section 1.3]{kesavan}, whose equimeasurability properties guarantee that
$$
\|\Delta v\|_{L^p(B_a)}=\|\Delta u\|_{L^p(\Omega_+)},\quad\int_{B_a}\Delta v=\int_{\Omega_+}\Delta u,
$$
and similar relations hold between $-u$ over $\Omega_-$ and $w$ over $B_b$. Then by elliptic regularity $v\in W^{2,p}(B_a)$ and $w\in W^{2,p}(B_b)$. By Talenti's comparison principle~\cite[Theorem 3.1.1 and Proposition 3.1.1]{kesavan},
\begin{equation}\label{eq:comparaison}
\begin{split}
u|_{\Omega_+}^*\leq v,\quad\|u\|_{L^p(\Omega_+)}\leq\|v\|_{L^p(B_a)},\quad\|\nabla u\|_{L^2(\Omega_+)}\leq\|\nabla v\|_{L^2(B_a)},
\end{split}
\end{equation}
and the same holds for $-u$ on $\Omega_-$ and $w$ on $B_b$. As a consequence,
$$
\frac{\int_\Omega|\Delta u|^p}{\int_\Omega|u|^p}\geq\frac{\int_{B_a}|\Delta v|^p+\int_{B_b}|\Delta w|^p}{\int_{B_a}|v|^p+\int_{B_b}|w|^p}.
$$
Note that since $\int_{B_a}\Delta v-\int_{B_b}\Delta w=\int_\Omega\Delta u$, by Green's identity $(v,w)$ is an admissible pair for $I_{a,b}$. Take the infimum first with respect to $(v,w)$ and then with respect to $u$ to obtain that $\Gamma_p(\Omega)\geq I_{a,b}$, where $a$ and $b$ may have changed since they depend upon the test function $u$ (here we also need the continuity of $I_{a,b}$ in $(a,b)$, see Lemma~\ref{lemme:continuité pb 2 boule} in Appendix). Note that $|B_a|+|B_b|\leq |\Omega|$ but rescaling the balls to saturate the constraint does not increase $I_{a,b}$, which concludes.
\end{proof}

We now show that $I_{a,b}$ is actually a minimum. Here and in what follows, we use the notation $[x]_p=|x|^{p-2}x$, which defines a bijective increasing function. In particular $\Delta_{2,p} u=\Delta[\Delta u]_p$.

\begin{lemme}\label{lemme:pb 2 boules}
The infimum in $I_{a,b}$ is achieved by a pair $(v,w)$ of nonnegative radial functions solving
\begin{equation}\label{eq:pb deux boules equation}
\left\{
\begin{array}{rcll}
\Delta_{2,p} v & = & I_{a,b} [v]_p& \text{in }B_a, \\
v & = & 0 & \text{on }\partial B_a,
\end{array}
\right.
\qquad
\left\{
\begin{array}{rcll}
\Delta_{2,p} w & = & I_{a,b} [w]_p & \text{in }B_b, \\
w & = & 0 & \text{on }\partial B_b,
\end{array}
\right.
\end{equation}
together with the coupling condition $\int_{\partial B_a}\partial_n v=\int_{\partial B_b}\partial_n w$.
\end{lemme}

\begin{remarque}\label{rmq:unicité}
We suspect, but do not claim, that the pair is unique, up to multiplicative constant.
\end{remarque}

\begin{proof}
By elliptic regularity a minimising sequence is bounded in $W^{2,p}$, hence the direct method of calculus of variations provides a minimising pair $(v,w)$ in $W^{2,p}$, vanishing at the boundary. The coupling condition between $\partial_n v$ and $\partial_n w$ follows from the fact that $(v,w)\mapsto\left(\int_{B_a}\Delta v,\int_{B_b}\Delta w\right)$  is continuous with respect to weak $W^{2,p}$ convergence. The first variation of the quotient in $I_{a,b}$ gives~\eqref{eq:pb deux boules equation} as Euler-Lagrange equations. Positivity is a consequence of the \emph{graph-flipping} argument given in \cite[equation (3.12)]{nadirashvili}: up to multiplicative constant, assume that $\partial_n v\leq 0$ and define the competitor $\tilde{v}$, vanishing at the boundary, by $\partial_r \tilde{v}=-|\partial_r v|$. Then $|\Delta \tilde{v}|=|\Delta v|$ while $\tilde{v}(r)=-\int_r^a\partial_r \tilde{v}\geq\left|\int_r^a\partial_r v\right|=|v(r)|$. Since $\tilde{v}$ has the same normal derivative as $v$, the pair $(\tilde{v},w)$ is admissible and does not increase the quotient, hence it is also a minimising pair. The same argument applies to $w$. Overall, the minimising pair may be assumed nonnegative.
\end{proof}

In the rest of the paper we study the depedence of the two-ball problem with respect to $a$ and $b$. Due to the volume constraint, we may assume by scaling that $a^d+b^d=1$. Our main result is to show that, in dimension two, the two-ball problem is minimal when one of the two balls vanishes.

\begin{theoreme}\label{thm:2-ball problem}
Let $0\leq a, b\leq 1$, such that $a^d+b^d=1$. If $d=2$, then $I_a:=I_{a,b}\geq I_{0,1}=I_{1,0}$.
\end{theoreme}

The proof is decomposed in two steps. The first one consists in estimating, in terms of $v$ and $w$, the variation of $I_a=I_{a,(1-a^d)^{1/d}}$ with respect to $a$. In the second step, we establish some pointwise comparison result between the minimisers $v$ and $w$, which allows to determin the sign of the variation.

\subsection{Variation of the two-ball problem}

Note that the problem is not obviously derivable with respect to $a$, especially when $p\neq 2$, making the question challenging. In~\cite{nadirashvili}, the author uses a transplantation technique, whose purpose is to explicitly construct competitors for $I_{a-\epsilon}$ from those of $I_a$. The method also works in the nonlinear setting, and gives an upper bound on the variation of the functional.

\begin{proposition}\label{prop:Ir croissante}
Let $0<a<b<1$ and $(v,w)$ be a minimising pair for $I_{a,b}$, made of nonnegative radial functions. Assume that
\begin{equation}\label{eq:condition pour la derivee}
a^{2-2d}\int_{B_a}v^{p-1}<b^{2-2d}\int_{B_b}w^{p-1}.
\end{equation}
Then $I_{a-\epsilon}<I_a$ for all small $\epsilon>0$.
\end{proposition}

\begin{remarque}
In the case $p=2,d=2$ equation~\eqref{eq:condition pour la derivee} boils down to \cite[equation (3.18)]{nadirashvili}.
\end{remarque}

\begin{proof}
Following the proof of~\cite[Lemma~3.19]{nadirashvili}, we consider a ball of radius $0<\alpha<a$, a ball of radius $b<\beta<1$ with $\alpha^d+\beta^d=1$, and we build a couple of admissible functions $(v_\alpha,w_\beta)$ for $I_{\alpha}$ in the following way.

Set $f:=\Delta v$ and $g:=\Delta w$. Let $f_\alpha$ be the restriction of $f$ on $B_{\alpha}$. On the other hand, let $g_\beta$ be an extension of $g$ on $B_{\beta}$ such that $g_\beta$ on the annulus $\{y:b<|y|<\beta\}$ is equimeasurable with $-f$ on the annulus $\{x:\alpha<|x|<a\}$. One possibility is to put
$$
g_\beta(y):=
\begin{cases}
g(y) & \text{if }y\in B_b, \\
(-f|_{\{\alpha<|x|<a\}})^\filledstar(|B_{|y|}|-|B_b|) & \text{if }y\in B_{\beta}\setminus B_b,
\end{cases}
$$
where $^\filledstar$ stands for the one-dimensional nonincreasing rearrangment (see \cite[Definition 1.1.1]{kesavan}). Thanks to the properties of equimeasurable functions, we immediately get
\begin{align*}
&\int_{B_{\alpha}}f_\alpha-\int_{B_{\beta}}g_\beta=\int_{B_a}f-\int_{B_b}g=\int_{\partial B_a}\partial_nv-\int_{\partial B_b} \partial_nw=0, \\
&\int_{B_{\alpha}}|f_\alpha|^p+\int_{B_{\beta}}|g_\beta|^p=\int_{B_a}|f|^p+\int_{B_b}|g|^p=C,
\end{align*}
where $C$ does not depend on $\alpha$ and $\beta$. Now  define $v_\alpha,w_\beta$ such that
\begin{equation}\label{eq:v1 tilde et v2 tilde}
\left\{
\begin{array}{rcll}
-\Delta v_\alpha & = & f_\alpha & \text{in }B_{\alpha}, \\
v_\alpha & = & 0 & \text{on }\partial B_{\alpha},
\end{array}
\right.
\qquad
\left\{
\begin{array}{rcll}
-\Delta w_\beta & = & g_\beta & \text{in }B_{\beta}, \\
w_\beta & = & 0 & \text{on }\partial B_{\beta},
\end{array}
\right.
\end{equation}
A main observation, based on the fact that not only $f_\alpha=f|_{B_\alpha}$ but also $g_\beta=g_1|_{B_\beta}$, is that
\begin{equation}\label{eq:expression de tilde v1 et tilde v2}
v_\alpha=v-v(\alpha),\qquad w_\beta=w_1-w_1(\beta),
\end{equation}
satisfy (\ref{eq:v1 tilde et v2 tilde}). Thanks to these expressions, one shall compute the derivatives of the $L^p$ norms of $v_\alpha$ and $w_\beta$ with respect to $\alpha$, under the constraint $\alpha^d+\beta^d=1$. Indeed, use radial coordinates and recall the rule of calculus $\frac{d}{d t}\int_0^t h(x,t)\mathrm{d}x =h(t,t)+ \int_0^t\partial_t h(x,t)\mathrm{d}x$. Together with (\ref{eq:expression de tilde v1 et tilde v2}) this gives (since $v$ is positive)
$$
\frac{d}{d\alpha}\left.\left[\int_{B_{\alpha}}|v_\alpha|^p\right]\right|_{\alpha=a}=-p\partial_n v\int_{B_a}v^{p-1}.
$$
Similarly, using also that $\frac{d}{d\alpha}\beta=-\frac{\alpha^{d-1}}{\beta^{d-1}}$, one obtains
$$
\frac{d}{d\alpha}\left.\left[\int_{B_{\beta}}|w_\beta|^p\right]\right|_{\alpha=a}=p\frac{a^{d-1}}{b^{d-1}}\partial_n w\int_{B_b}w^{p-1}.
$$
As a consequence,
$$
a^{1-d}\frac{d}{d\alpha}\left.\left[\int_{B_\alpha}|v_\alpha|^p+\int_{B_\beta}|w_\beta|^p\right]\right|_{\alpha=a}=pa^{2-2d}\int_{B_a}v^{p-1}\int_{\partial B_a}\partial_n v-pb^{2-2d}\int_{B_b}w^{p-1}\int_{\partial B_b}\partial_n w< 0,
$$
where the last inequality comes from our assumption~\eqref{eq:condition pour la derivee} and from the fact that $\int_{\partial B_a}\partial_n v=\int_{\partial B_b}\partial_n w$, since $(v,w)$ is admissible. Then, for $0<\alpha<a$ sufficiently close to $a$, we have
$$
\frac{\int_{B_\alpha}|\Delta v_\alpha|^p+\int_{B_\beta}|\Delta w_\beta|^p}{\int_{B_\alpha}|v_\alpha|^p+\int_{B_\beta}|w_\beta|^p}< I_a,
$$
and since $(v_\alpha,w_\beta)$ is admissible for $I_\alpha$,
$$
I_\alpha< I_a,
$$
which concludes.
\end{proof}

\begin{remarque}
In the linear case $p=2$, condition~\eqref{eq:condition pour la derivee} is equivalent to $a^{1-d}\partial_n\Delta v<b^{1-d}\partial_n \Delta w$, as can be seen after integrating equation~\eqref{eq:pb deux boules equation}. This last relation can also be recovered from a shape-derivative approach, as long as the minimisers $v$ and $w$ of $I_a$ are assumed to be derivable with respect to $a$. Indeed if $\dot{v},\dot{w},\dot{I}$ denote the derivatives of $v,w,I_a$ respectively, one verifies that
$$
\left\{
\begin{array}{rcll}
\Delta^2 \dot{v} & = & I_a\dot{v}+\dot{I}v  & \text{in }B_{a}, \\
v & = & 0 & \text{on }\partial B_a,\\
\partial_n v & = & -\Delta v &  \text{on }\partial B_a,
\end{array}
\right.
\qquad
\left\{
\begin{array}{rcll}
\Delta^2 \dot{w} & = & I_a\dot{w}+\dot{I}w  & \text{in }B_{b}, \\
w & = & 0 & \text{on }\partial B_b,\\
\partial_n w & = & \frac{a^{d-1}}{b^{d-1}}\Delta w  & \text{on }\partial B_b,
\end{array}
\right.
$$
from which one obtains the formula for $\dot{I}$ after integration by parts. That such an approach would work also in the nonlinear case seems non-trivial -- especially concerning the derivability of the problem in $a$.
\end{remarque}

\subsection{Comparison between the minimisers}\label{sec:comparaison}

Using Proposition \ref{prop:Ir croissante}, we see that Theorem \ref{thm:2-ball problem} will hold if one is able to show that $a^{2-d}V^{p-1}$ is pointwise smaller than $b^{2-d}W^{p-1}$ on the unit ball $B$, where $v$ and $w$ minimise $I_a$ and $V$ and $W$ are their dilation on $B$ -- namely $V(x)=v(ax)$ and $W(y)=w(by)$. Those functions satisfy
\begin{equation}\label{eq:Euler-Lagrange}
\left\{
\begin{array}{rcll}
\Delta_{2,p} V & = & a^{2p}I_a |V|^{p-2}V & \text{in }B, \\
V & = & 0 & \text{on }\partial B,
\end{array}
\right.
\qquad
\left\{
\begin{array}{rcll}
\Delta_{2,p} W & = & b^{2p}I_a |W|^{p-2}W & \text{in }B, \\
W & = & 0 & \text{on }\partial B,
\end{array}
\right.
\end{equation}
and since $\int_{\partial B_a}\partial_nv=\int_{\partial B_b}\partial_nw$, we also have
\begin{equation}\label{eq:égalité des dérivées normales}
a^{d-2}\partial_nV=b^{d-2}\partial_nW\text{ on }\partial B.
\end{equation}
Here the planeness becomes crucial since for $d>2$, inequality~\eqref{eq:égalité des dérivées normales} implies $|\partial_n V|>|\partial_n W|$ and hence $a^{2-d}V^{p-1}>b^{2-d}W^{p-1}$ close to the boundary, in contradiction with the inequality we seek to establish. Yet what can be proved is that $a^{d-2}V < b^{d-2}W$. In dimension two, this is equivalent to the desired inequality.

\begin{proposition}\label{prop:comparaison}
For $0<a<b$, let $V,W>0$ be radial and solve \eqref{eq:Euler-Lagrange}--\eqref{eq:égalité des dérivées normales}. Then, $a^{d-2}V<b^{d-2}W$.
\end{proposition}

\begin{proof}
The proof follows by an induction based on the maximum principle. Indeed, consider some large $\gamma>0$, and observe that $\Delta_{2,p}(a^{\gamma-2p/(p-1)}V)=I_a [a^\gamma V]_p$, and similarly $\Delta_{2,p}(b^{\gamma-2p/(p-1)}V)=I_a [b^\gamma V]_p$. Therefore,
\begin{equation}\label{eq:application principe du max}
\Delta_{2,p}(a^{\gamma-2p/(p-1)}V)-\Delta_{2,p}(b^{\gamma-2p/(p-1)}W)=I_a[a^\gamma V]_p-I_a[b^\gamma W]_p\underset{\gamma\to\infty}{\sim} -I_a[b^\gamma W]_p<0.
\end{equation}
The convergence holds uniformly and comes from $a<b$, from the positivity of $V,W$, and from~\eqref{eq:égalité des dérivées normales}. In particular, for any large enough $\gamma$ one has $[a^\gamma V]_p<[b^\gamma W]_p$ in $B$, which initializes the induction. On the other hand, at the boundary one has $a^{\gamma-2p/(p-1)}V=b^{\gamma-2p/(p-1)}W=0$; while by~\eqref{eq:égalité des dérivées normales} and by nonpositivity of the normal derivatives
$$
\partial_n(a^{\gamma-2p/(p-1)}V)\geq \partial_n(b^{\gamma-2p/(p-1)}W),$$
as long as $\gamma\geq d+2/(p-1)$. The strong maximum principle, Lemma~\ref{lemme:principe du max}, applied to~\eqref{eq:application principe du max} then transforms the relation $[a^\gamma V]_p<[b^\gamma W]_p$ into $[a^{\gamma-2p/(p-1)} V]_p<[b^{\gamma-2p/(p-1)} W]_p$, which closes the induction. Applying the result to $\gamma=d+2/(p-1)$ gives the desired result.
\end{proof}

\subsection{Conclusions}

\begin{proof}[Proof of Theorem~\ref{thm:2-ball problem}]
Take $d=2$, $0<a<1/\sqrt{2}<b<1$, and consider a minimising pair $(v,w)$ for $I_{a,b}$ with $v,w\geq0$, by Lemma~\ref{lemme:pb 2 boules}. If $\partial_n v=\partial_n w=0$, then $I_{a,b}=\Gamma_p(B_a\sqcup B_b)>\Gamma_p(B_1)=I_{0,1}=I_{1,0}$, as desired. Otherwise $\partial_nv,\partial_nw<0$ and hence $v$ and $w$ are strictly positive. Dilate $v$ and $w$ on the unit ball, as explained in section~\ref{sec:comparaison}, to obtain $V>0$ and $W>0$. By Proposition~\ref{prop:comparaison} we have $V<W$. In particular $\int_B V^{p-1}<\int_B W^{p-1}$, which amounts to~\eqref{eq:condition pour la derivee} in dimension two. By Proposition~\ref{prop:Ir croissante} we conclude that $I_a$ decreases if we slightly reduce $a$. Together with the continuity of $I_a$ with respect to $a$ (see Lemma~\ref{lemme:continuité pb 2 boule}) this shows that the minimum is achieved if and only if $a=0$.
\end{proof}

\begin{proof}[Proof of Theorem~\ref{thm:Faber-Krahn} and Theorem~\ref{thm:cas de la boule}]
Consider without loss of generality a planar domain $\Omega$ of same area as the unit disk $\mathbb{D}$. By Proposition~\ref{prop:extraction pb 2 boules} and Theorem~\ref{thm:2-ball problem} we have
$\Gamma_p(\Omega)\geq I_{a,b}\geq I_{1,0}=\Gamma_p(\mathbb{D})$. If furthermore $\Omega$ has the same eigenvalue as $\mathbb{D}$ and admits an eigenfunction $u$. Recall from~\eqref{eq:comparaison} the inequalities $0<u_+^*\leq v$ and $\int_{B_a}|\nabla u_+^*|^2\leq\int_{\Omega_+}|\nabla u_+|^2\leq\int_{B_a}|\nabla v|^2$, where $v$ is the competitor obtained from $u_+$ after Talenti's symmetrisation, assumed non-identically zero without loss of generality. The chain of inequalities is now
$$
\frac{\int_\Omega|\Delta u|^p}{\int_\Omega |u|^p}\geq\frac{\int_{B_a}|\Delta v|^p+\int_{B_b}|\Delta w|^p}{\int_{B_a}|v|^p+\int_{B_b}|w|^p}\geq I_{a,b}\geq I_{1,0}=\frac{\int_\Omega|\Delta u|^p}{\int_\Omega |u|^p},
$$
hence all inequalities are equalities. In particular $a=1$ and $v=u_+^*>0$ is a minimiser of $I_{1,0}$. Then $u_+^*$ cannot be locally constant -- otherwise $\Delta_{2,p} v=I_{1,0} [v]_p>0$ would fail somewhere inside $B_a$ -- hence $u_+^*$ is strictly decreasing. Together with $\int_{B_a}|\nabla u_+^*|^2=\int_{\Omega_+}|\nabla u_+|^2$, this means by~\cite[Theorem 2.3.3]{kesavan} that $u_+=u_+^*$ a.e., and since $|\Omega_-|=0$ we conclude that $u$ is a.e. positive, radially symmetric and decreasing. Since in dimension two $u\in W^{2,p}(\R^2)$ is continuous, we conclude that $\Omega=\mathbb{D}$, up to isometries. The previous discussion applies when $\Omega$ is the disk, which admits an eigenfunction by regularity, and shows that this eigenfunction must be positive, radially symmetric and decreasing.
\end{proof}

\appendix
\section*{Appendix}

\begin{lemme}\label{lemme:continuité pb 2 boule}
$I_{a,b}$ is continuous with respect to $(a,b)\in[0,\infty)^2\setminus\{(0,0)\}$.
\end{lemme}

\noindent This is a generalisation of~\cite[Lemma~3.20]{nadirashvili}, yet the proof differs.

\begin{proof}
Consider first the case where $a,b\neq 0$. After a change of variable, $I_{a,b}$ can be seen as the infimum of the functional $
F_{a,b}(V,W):=[a^{-d}\int_{B_1}|\Delta V|^p+b^{-d}\int_{B_1}|\Delta W|^p]/[a^{4-d}\int_{B_1}|V|^p+b^{4-d}\int_{B_1}|W|^p]$,
over the fixed subspace of $W^{2,p}(B_1)^2$ made of pairs of radial functions $(V,W)$ vanishing at the boundary, and such that $\partial_n V=\partial_n W$. Since the functional is continuous in $(a,b)$, by standard results the infimum is upper semi-continuous in $(a,b)$. For lower semi-continuity, observe that $I_{a,b}\leq\Gamma_p(B_a\sqcup B_b)$ since an admissible function for the latter gives rise to an admissible pair for the former, hence away from the origin $I_{a,b}$ is bounded. As a consequence if $(a_n,b_n)\to(a,b)$ and if $(V_n,W_n)$ is a normalised minimising pair for $I_{a_n,b_n}$, by elliptic regularity $V_n$ and $W_n$ are bounded in $W^{2,p}(B_1)$ and hence weakly converge, up to subsequence, towards $V$ and $W$. The pair $(V,W)$ is admissible, as can be seen from passing to the limit in $\int_{B_1}\Delta V_n=\int_{B_1}\Delta W_n$, hence $F_{a,b}(V,W)\geq I_{a,b}$. Taking the liminf in $F_{a_n,b_n}(V_n,W_n)$ then gives lower semi-continuity.

In the case where, say, $b=0$, one cannot resort to $F_{a,b}$, hence we directly consider a minimiser $v\in W_0^{2,p}(B_a)$ of $I_{a,0}$. If $(a_n,b_n)\to(a,0)$, the pair $(v(ax/a_n),0)$ is admissible for $I_{a_n,b_n}$. Taking the $\limsup$ in the associated quotient gives upper semi-continuity. For lower semi-continuity, consider a normalised minimising pair $(v_n,w_n)$ for $I_{a_n,b_n}$. As previously, $v_n$ and $w_n$ have their Laplacian bounded in $L^p$, hence $\tilde{v}_n(x)=v_n(a_nx/a)$ converges weakly towards some $v$ in $W^{2,p}(B_a)$. Furthermore by Hölder inequality $\int_{B_a}\Delta \tilde{v}_n=a^{d-2}/a_n^{d-2}\int_{B_{b_n}}\Delta w_n=O(b_n^{d(1-1/p)})$, hence after passing to the limit $v\in W_0^{2,p}(B_a)$ and we may use it as test function in $I_{a,0}$. On the other hand, $w_n\to 0$ in $L^p$ since $\|w_n\|_{L^p(B_{b_n})}\leq C(b_n)\|\Delta w_n\|_{L^p(B_{b_n})}$, where $C(b_n)$ is the operator norm for the injection of $W^{2,p}\cap W_0^{1,p}(B_{b_n})$ in $L^p(B_{b_n})$, which scales like $b_n^2$. Taking the $\liminf$ in the quotient of $(v_n,w_n)$, after having dropped $\int_{B_{b_n}}|\Delta w_n|^p$, gives lower semi-continuity.
\end{proof}

\begin{lemme}[Maximum principle for radial functions]\label{lemme:principe du max}
Let $f,g:B\to\R$ be $C^2(\overline{B})$ radial functions such that $f\geq g$ on $\partial B$ and $\partial_n f\leq \partial_n g$ on $\partial B$. Then,
$$
\Delta_{2,p} f>\Delta_{2,p} g\text{ in }B\qquad \Longrightarrow \qquad f>g\text{ in }B.
$$
\end{lemme}

\noindent In the linear case, this result was already know by \cite{laurençot-walker}.

\begin{proof}
Assume by contradiction that $h:=f-g$ is not positive. Let $r_0$ be  the radius of a circle on which $h$ reaches its minimal value (recall that $h$ is radially symmetric). Then, $h(r_0)\leq 0$, $\nabla h(r_0)=0$, and $\Delta h(r_0)\geq0$.

We show that $\Delta h$ is not positive on the annulus $\{r_0<|x|<1\}$. Indeed, otherwise, by the maximum principle, $h$ would reach its maximum at the boundary of the annulus.
But since $h(r_0)\leq 0\leq h(1)$, the maximum is attained at the boundary of $B$. Then, by Hopf boundary Lemma, $\partial_n h>0$ on $\partial B$, which is impossible.

Therefore, $\Delta h$ takes nonpositive values in $\{r_0< |x|<1\}$. But since $\Delta h(r_0)\geq 0$, by continuity, $\Delta h$ vanishes somewhere in $\{r_0\leq|x|<1\}$.
Let us denote $R_0\in[r_0,1]$ the greatest radius of a circle on which $\Delta h$ vanishes, that is $R_0=\sup\{r\in[r_0,1[:\Delta h(r)=0\}$.

Since $\Delta ([\Delta f]_p-[\Delta g]_p)>0$ in the ball $\{|x|<R_0\}$, and since $[\Delta f]_p=[\Delta g]_p$ on its boundary, we apply the strong maximum principle and obtain that $[\Delta f]_p<[\Delta g]_p$, and hence that $\Delta h<0$ in $B_{R_0}$. We deduce in the following that $R_0>r_0$.

Assume first that $r_0>0$. If we had $R_0=r_0$ we would get $\nabla h(R_0)=\nabla h(r_0)=0$, but Hopf boundary Lemma applied in $B_{R_0}$ would lead to the contradiction $|\nabla h(R_0)|=|\partial_n h(R_0)|>0$. Assume instead that $0=r_0=R_0$, so that $h(0)\leq0$ and $\nabla h(0)=0$. Then, due to the definition of $R_0$, $\Delta h\neq0$ in $\{0<|x|<1\}$. If $\Delta h<0$, since $h(0)\leq0\leq h(1)$ Hopf Lemma would show that $\nabla h(0)\neq0$, a contradiction; and if $\Delta h>0$, it would give $\partial_n h>0$, a contradiction as well.

Now, applying the strong maximum principle in $B_{R_0}$ shows that the minimum of $v$ is attained at the boundary. In particular we have $h(R_0)<h(r_0)$, a contradiction since $h$ was supposed to reach its minimum at $r_0$.
\end{proof}

\section*{Acknowledgements}
I want to thank Simone Mauro for some interesting discussions on the subject. This work was partially supported by Aix-Marseille Université, CNRS, I2M, Marseille (France) and by the Funda\c{c}\~ao para a Ci\^encia e a Tecnologia, Portugal, via the research centre
GFM, reference UID/00208/2025.



\printbibliography



\end{document}